\documentclass[11pt]{article}
\usepackage{amsfonts,mathtools}
\usepackage{amsmath,amsthm, bbm}
\usepackage[usenames,dvipsnames]{color}
\usepackage{hyperref}

\newcommand{\beq}{\begin{equation}}
\newcommand{\eeq}{\end{equation}}

\newcommand{\bea}{\begin{eqnarray}}
\newcommand{\eea}{\end{eqnarray}}
\newcommand{\beas}{\begin{eqnarray*}}
\newcommand{\eeas}{\end{eqnarray*}}

\newtheorem{theorem}{Theorem}[section]

\newtheorem{proposition}[theorem]{Proposition}

\newtheorem{lemma}[theorem]{Lemma}
\newtheorem{remark}[theorem]{Remark}
\newtheorem{example}[theorem]{Example}
\newtheorem{examples}[theorem]{Examples}
\newtheorem{foo}[theorem]{Remarks}

\newcommand{\ep}{\epsilon}

\newcommand{\Hh}{\hat{\mathcal{H}}}
\newcommand{\E}{\mathbb E}
\newcommand{\C}{\mathbb C}
\newcommand{\U}{\mathbf U}
\newcommand{\bB}{\mathbf B}

\newcommand{\tr}{\mathrm tr}

\title{Asymptotic windings of the block determinants of a unitary Brownian motion and related diffusions}
\author{Fabrice Baudoin\footnote{F.B. was partly supported by the NSF grant DMS~1901315.},  Jing Wang\footnote{J.W. was partly supported by the NSF grant DMS~1855523.} 
}

\usepackage{todonotes}
\mathtoolsset{showonlyrefs=true}
\begin{document}
\maketitle

\begin{abstract}
 We study several matrix diffusion processes constructed from a unitary Brownian motion.  In particular, we  use the Stiefel fibration to lift the Brownian motion of the complex Grassmannian to the complex Stiefel manifold and deduce a skew-product decomposition of the Stiefel Brownian motion. As an application,  we prove  asymptotic  laws for the determinants of the block entries of the unitary Brownian motion.
\end{abstract}

\tableofcontents

\section{Introduction}

 The  Brownian motion on the Lie group of complex unitary matrices, in short unitary Brownian motion, is an extensively studied object in random matrix theory (see for instance \cite{MR3748321, MR2407946} and references therein).  In this paper we study several diffusion processes naturally associated to this Brownian motion and show how they can be used to compute the exact or asymptotic distribution  of some functionals. In particular, we emphasize the role of the Stiefel fibration over the complex Grassmannian space as an effective computational tool. As an application of our methods we obtain for instance the following result.

\begin{theorem}\label{winding intro}
Let
$$U_t=\begin{pmatrix}  X_t & Y_t \\ Z_t & W_t \end{pmatrix}$$
 be a Brownian motion on the unitary group $\mathbf{U}(n)$ with  $Z_t \in \mathbb{C}^{k\times k}$, $1\le k \le n-1$. Assume that $\det Z_0 \neq 0$. One has then the polar decomposition
\[
\det (Z_t)=\varrho_t e^{i\theta_t}
\]
where $0< \varrho_t \le 1$ is a continuous semimartingale such that, in distribution, when $t \to +\infty$,
\[
\varrho_t^2 \to \prod_{j=1}^{\min (k,n-k)} \mathfrak{B}_{j,\max (k,n-k)}
\]
where $\mathfrak{B}_{a,b}$ are independent beta random variables with parameters $(a,b)$ and $\theta_t$ is a real-valued continuous martingale such that the following convergence holds in distribution when $t \to +\infty$
\[
\frac{\theta_t}{t} \to \mathcal{C}_{k(n-k)},
\]
where $\mathcal{C}_{k(n-k)}$ is a Cauchy distribution of parameter $k(n-k)$.
\end{theorem}

The study of the limit law of $\varrho_t$ can be deduced from known results about the complex Jacobi ensemble, see \cite{MR2365642}. Therefore, our main contribution lies in the study of the winding process $\theta_t$. We note that for $k=1$, the theorem therefore yields a Spitzer's type theorem (see \cite{MR104296}) for the windings of each of the entries of a unitary Brownian motion (the choice of bottom left corner block  for $Z$ is  arbitrary). Overall, the proof of Theorem \ref{winding intro}  is rather long.  An important step, interesting in itself, is to obtain the  representation 
\[
  \theta_t = \theta_0+ i \mathrm{tr}(D_t)+ \int_{w[0,t]} \alpha
\]
where $D_t$ is a Brownian motion on the Lie algebra $\mathfrak{u}(k)$, $w$ is a Brownian motion on the complex Grassmannian manifold $G_{n,k}$ 
and  $\alpha$ is an almost everywhere defined one-form on $G_{n,k}$ whose exterior derivative yields the K\"ahler form of $G_{n,k}$. To study the functional $\int_{w[0,t]} \alpha$ which might be interpreted as a generalized stochastic area process in the sense of \cite{baudoin2017stochastic}, we prove that $\int_{w[0,t]} \alpha$ is a time-changed Brownian motion with independent clock $\int_0^t \mathrm{tr} (w_s^* w_s) ds$ and finally study the distribution of this latter additive functional using a  Girsanov transform in the spirit of  \cite{baudoin2017stochastic, doumerc2005matrices,  Yor1}.

The paper is structured as follows. Section 2 mostly collects preliminary results. We  present  the geometry of  the Stiefel fibration over the complex Grassmannian space, introduce a convenient set of coordinates and show how to obtain representations of the Brownian motions on the Stiefel and complex Grassmannian spaces from a unitary Brownian motion, see Theorem \ref{main:s1}. We then study some related eigenvalue processes. In Section 3, we take advantage of the Stiefel fibration over the complex Grassmannian space to explicitly write a skew-product decomposition of the Stiefel Brownian motion. Section 4 is devoted to the study of limit laws of some functionals of the unitary Brownian motion and culminates with the proof of Theorem \ref{winding intro}

\

\textbf{Notations:}

\begin{itemize}

\item If $M \in \mathbb{C}^{n \times n}$ is a $n \times n$ matrix with complex entries, we will denote $M^*=\overline{M}^T$ its adjoint.

\item If $z_i=x_i +i y_i$ is a complex coordinate system, let
\[
\frac{\partial }{\partial z_i }= \frac{1}{2} \left(  \frac{\partial }{\partial x_i } -i \frac{\partial }{\partial y_i }\right) , \quad \frac{\partial }{\partial \bar{z}_i }= \frac{1}{2} \left(  \frac{\partial }{\partial x_i } +i \frac{\partial }{\partial y_i }\right).
\]

\item Throughout the paper we work on a filtered probability space $(\Omega, (\mathcal{F}_t)_{t \ge 0} , P)$ that satisfies the usual conditions.

\item If $X$ and $Y$ are semimartingales, we denote $\int X dY$ the It\^o integral, $\int X \circ dY$ the Stratonovich integral and $\int dX dY$ or $\langle X, Y \rangle$ the quadratic covariation.

\item For matrix-valued semimartingales $M$ and $N$, the quadratic variation $\int dMdN$ is a matrix such that $\left(\int dMdN\right)_{ij}= \sum_{\ell} \int dM_{i\ell}dN_{\ell j}$.

\item If $M$ is a semimartingale and $\eta$ a one-form, then $\int_{M[0,t]} \eta$ denotes the Stratonovich line integral of $\eta$ along the paths of $M$.

\end{itemize}

\section{Some diffusions related to unitary Brownian motion}\label{Sec2}

\subsection{Stiefel fibration} \label{fibration}

Let $n\in \mathbb{N}$, $n \ge 2$, and $k\in\{1,\dots, n\}$. The complex Stiefel manifold $V_{n,k}$ is the set of  unitary $k$-frames  in $\C^n$. In matrix notation we have
\[
V_{n,k}=\{M \in \C^{n\times k}| M^* M=I_k\}.
\] 

As such $V_{n,k}$ is therefore an algebraic compact embedded submanifold of $\C^{n\times k}$ and  inherits from  $\C^{n\times k}$ a Riemannian  structure. We note that $V_{n,1}$ is isometric to the unit sphere  $\mathbb{S}^{2n-1}$. There is a right isometric action of the  unitary group $\mathbf{U}(k)$ on $V_{n,k}$, which is simply given by the right matrix multiplication: $M  g$, $M \in V_{n,k}$,  $g \in \mathbf{U}(k)$. The quotient space by this action $G_{n,k}:=V_{n,k} / \mathbf{U}(k)$ is the complex Grassmannian manifold.  It is a compact manifold of complex dimension $k(n-k)$. We note that $G_{n,k}$ can be identified with the set of $k$-dimensional subspaces of $\mathbb{C}^n$. In particular $G_{n,1}$ is the complex projective space $\mathbb{C}P^{n-1}$. Since $G_{n,k}$ and $G_{n,n-k}$ can be identified with each other via orthogonal complement, without loss of generality we will therefore assume throughout the paper that \fbox{$k \le n-k$}, except when explicitly mentioned otherwise (see section \ref{duality section}). 

Let us quickly comment on the Riemannian structure of $G_{n,k}$ that we will be using and that is induced from the one of $V_{n,k}$.  From Example 2.3 in \cite{BIHP} , there exists a unique Riemannian metric on $G_{n,k}$ such that the projection map $\pi: V_{n,k} \to G_{n,k}$ is a Riemannian submersion. From Example 2.5 in \cite{BIHP} and Theorem 9.80 in \cite{Besse} the fibers of this submersion are totally geodesic submanifolds of $V_{n,k}$ which are isometric to $\mathbf{U}(k)$. This therefore yields a fibration:
\[
\mathbf{U}(k) \to V_{n,k} \to G_{n,k}
\]
which is often referred to as the Stiefel fibration, see  \cite{autenried2014sub, jurdjevic2019extremal}.  We note that for $k=1$ it is nothing else but the classical  Hopf fibration considered from the probabilistic viewpoint in \cite{baudoin2017stochastic}:
\[
\mathbf{U}(1) \to \mathbb{S}^{2n-1} \to \mathbb{C}P^{n-1}.
\]
For further details on the Riemannian geometry of the complex Grassmannian manifolds we also refer to \cite{wong1967differential,wong1968sectional}, see in particular Theorem 4 in \cite{wong1967differential}. 
%

\subsection{Inhomogeneous coordinates on $G_{n,k}$}\label{inhomogeneous}

We consider  the open set $\widehat{V}_{n,k} \subset V_{n,k}$ given by
\[
\widehat{V}_{n,k}= \left\{ \begin{pmatrix} X  \\ Z \end{pmatrix}  \in  V_{n,k} , \, \det Z \neq 0 \right\}
\]
and the smooth map $p : \widehat{V}_{n,k} \to \C^{(n-k)\times k}$ given by $p \begin{pmatrix} X  \\ Z \end{pmatrix}  = X Z^{-1}.$ It is clear that for every $g \in \mathbf{U}(k)$ and $M \in V_{n,k}$, $p( M  g)= p(M)$. Since $p$ is a submersion from $\widehat{V}_{n,k}$ onto its image $ p ( \widehat{V}_{n,k}) = \C^{(n-k)\times k} $ we deduce that there exists a unique  diffeomorphism $\Psi$ from  an open set of $G_{n,k}$ onto $\C^{(n-k)\times k}$ such that 
\begin{align}\label{inh-coord}
 \Psi \circ \pi= p.
 \end{align}
 The map $\Psi$ induces a (local) coordinate chart on $G_{n,k}$ that we call inhomogeneous by analogy with the case $k=1$ which corresponds to the complex projective space. The Riemannian metric of $G_{n,k}$ is then transported to $\C^{(n-k)\times k}$ using the map $\Psi$. In the sequel, we will denote $\C^{(n-k)\times k}$ endowed with this Riemannian metric by $\widehat{G}_{n,k}$ in order to emphasize the Riemannian structure that is used. Note that by construction $\widehat{G}_{n,k}$ is isometric to an open  subset of $G_{n,k}$ and that it \textit{differs} from $G_{n,k}$ by the sub-manifold at $\infty$, $\det (Z)=0$. In the case $n=2, k=1$, $G_{n,k}=\mathbb{CP}^{1}$ and the above description corresponds to the classical one-point compactification description $\mathbb{CP}^{1}=\mathbb{C} \cup \{ \infty \}$. We note that the Stiefel fibration
 \[
\mathbf{U}(k) \to V_{n,k} \to G_{n,k}
\]
yields a fibration
\[
\mathbf{U}(k) \to\widehat{V}_{n,k} \to \widehat{G}_{n,k}
\]
that we still refer to as the Stiefel fibration. The projection map $p:\widehat{V}_{n,k} \to \widehat{G}_{n,k}$,  $ \begin{pmatrix} X  \\ Z \end{pmatrix}  \to  X Z^{-1}$ is then a Riemannian submersion with totally geodesic fibers isometric to $\mathbf{U}(k)$.

\subsection{Brownian motions on $\widehat{G}_{n,k}$ and $V_{n,k}$}

In this section, we show how the Brownian motions on $\widehat{G}_{n,k}$ and $V_{n,k}$ can be constructed from a Brownian motion on the unitary group $\mathbf{U}(n)$. In the following, we will  use the block notations as below: For any $U\in \mathbf{U}(n)$ and $A\in \mathfrak{u}(n)$ we will write
\[
U=\begin{pmatrix}  X & Y \\ Z & V \end{pmatrix}, \quad A=\begin{pmatrix}  \alpha & \beta \\ \gamma & \ep \end{pmatrix}
\]
where $X,\gamma \in \C^{(n-k)\times k}$, $Y,\ep\in \C^{(n-k)\times (n-k)}$, $Z,\alpha \in \C^{k\times k}$, $V,\beta\in \C^{k\times (n-k)}$.  We recall that the Lie algebra $\mathfrak{u}(n)$ consists of all skew-Hermitian matrices
\[
\mathfrak{u}(n)=\{X\in \C^{n\times n}|X=-X^*\},
\]
which we equip  with the inner product $\langle X,Y\rangle_{\mathfrak{u}(n)}=-\frac12\mathrm{tr}(XY)$. This induces a Riemannian metric on $\U(n)$.   Consider now on $\mathfrak{u}(n)$ a Brownian motion $(A_t)_{t \ge 0}$ and  the matrix-valued process  $(U_t)_{t \ge 0}$ that satisfy the Stratonovich stochastic differential equation:
\begin{equation}\label{eq-du-da}
\begin{cases}
dU_t=U_t\circ dA_t ,\\
U_0=\begin{pmatrix}  X_0 & Y_0 \\ Z_0 & V_0 \end{pmatrix}, \, \det Z_0 \neq 0.
\end{cases}
\end{equation}
The process  $(U_t)_{t \ge 0}$ is a Brownian motion on $\mathbf{U}(n)$ (which is not started from the identity).  If we write the block decomposition
\[
 U_t=\begin{pmatrix}  X_t & Y_t \\ Z_t & V_t \end{pmatrix}
 \]
 then, it is known that both of the processes $X^*X$ and $Z^* Z=I_k-X^*X$ belong to the well-known family of (complex) matrix Jacobi processes that have already been extensively  studied in the literature, see for instance \cite{MR3842169}, \cite{demni2010beta},  \cite{doumerc2005matrices}, \cite{graczyk2013multidimensional} and \cite{graczyk2014strong}. In particular, one can see that $Z^*Z$ satisfies 
 \[
 d( Z^*Z) =  \sqrt{Z^*Z}d\bB\sqrt{I_k-Z^*Z}+\sqrt{I_k-Z^*Z}d\bB^* \sqrt{Z^*Z}+\bigg(2kI_k-2nZ^*Z \bigg)dt
\]
where $\bB$ is a $\mathbb{C}^{k \times k}$ Brownian motion. 

\begin{theorem}\label{main:s1}
Let $U_t=\begin{pmatrix}  X_t & Y_t \\ Z_t & V_t \end{pmatrix}$, $t\ge0$ be the solution of \eqref{eq-du-da}.  
\begin{enumerate}
\item The process $\begin{pmatrix}  X_t  \\ Z_t  \end{pmatrix}_{t\ge0}$ is a Brownian motion on $V_{n,k}$;

\item We have $\mathbb{P} \left( \inf\{t>0, \det Z_t=0\} <+\infty \right)=0$  and the process $(w_t)_{t \ge 0}:=(X_tZ_t^{-1})_{t \ge 0}$ is a Brownian motion on $\widehat{G}_{n,k}$  with  generator  $\frac12\Delta_{\widehat{G}_{n,k}}$, where 
\begin{align*}
\Delta_{\widehat{G}_{n,k}}&={{4}}\sum_{1\le i, i'\le n-k, 1\le j, j'\le k}(I_{n-k}+ww^*)_{ii'} (I_k+w^*w)_{j'j}\frac{\partial^2}{\partial w_{ij}\partial \bar{w}_{i'j'}}.
\end{align*}

\end{enumerate}

\end{theorem}

\begin{proof}
The first part of the theorem is straightforward to prove. Indeed, the  map $\Pi: \mathbf{U}(n) \to V_{n,k}$, $\begin{pmatrix}  X & Y  \\ Z  & V  \end{pmatrix} \to \begin{pmatrix}  X   \\ Z    \end{pmatrix}  $ is a Riemannian submersion with totally geodesic fibers, therefore the process $(\Pi(U_t))_{t \ge 0}$ is a Brownian motion on $V_{n,k}$. We now turn to the proof of the second part of the theorem. We first note that, as noticed above,  $(Z^*_t Z_t)_{ t \ge 0}$ is a matrix Jacobi process and therefore, from known properties of those processes, $\mathbb{P} \left( \inf\{t>0, \det (Z^*_tZ_t)=0\} <+\infty \right)=0$. We then turn to the study of $w=XZ^{-1}$ which is therefore well defined for all times.  We need to introduce some notations. Let us consider the block decomposition
\[
A_t=\begin{pmatrix}  \alpha_t & \beta_t \\ \gamma_t & \ep_t \end{pmatrix},
\]
with $\alpha_t \in \mathbb{C}^{k \times k}$. Note that $\alpha_t$, $\beta_t=-\gamma_t^*$ and $\ep_t$ are independent. 
From \eqref{eq-du-da} we  obtain the following system of stochastic differential equations:
\begin{align}\label{eq-matrix-sde}
dX &=X \circ d\alpha +Y \circ d\gamma =X  d\alpha +Y d\gamma +\frac12 (dX  d\alpha +dY  d\gamma )\nonumber \\
dY &=X \circ d\beta +Y \circ d\ep =X  d\beta +Y  d\ep +\frac12 (dX  d\beta +dY  d\ep)\nonumber\\
dZ &= Z \circ d\alpha +V \circ d\gamma =Z  d\alpha +V d\gamma +\frac12 (dZ  d\alpha +dV  d\gamma )\\
dV &=Z \circ d\beta +V \circ d\ep =Z  d\beta +V  d\ep +\frac12 (dZ  d\beta +dV  d\ep ).\nonumber
\end{align}
Using then It\^o's formula, long but routine computations yield
\begin{align}\label{w dynamics}
dw=d(XZ^{-1}
)=(Y d\gamma -w V d\gamma) Z^{-1}.
\end{align}
From this expression, after further computations, we deduce
\begin{align*}
dw_{ij}d\bar{w}_{i'j'}&=\sum_{\ell,m=1}^k(Y-w V )_{i\ell}(\overline{Y}-\overline{w V} )_{i'm} (d\gamma Z^{-1})_{\ell j}(d\overline{\gamma} \overline{Z}^{-1})_{m j'} \\
 &=2 ((Y-w V )(Y-w V)^* )_{ii'}((ZZ^*)^{-1})_{j'j}dt \\
 &=2(I_{n-k}+ww^*)_{ii'}((ZZ^*)^{-1})_{j'j}dt \\
&=2(I_{n-k}+ww^*)_{ii'}(I_k+w^*w)_{j'j}dt.
\end{align*}
Therefore $(w_t)_{t \ge 0}$ is a diffusion with generator $\frac12\Delta_{\widehat{G}_{n,k}}$. Then, to conclude, we note that it is indeed the Brownian motion on $\widehat{G}_{n,k}$ because $p$ is a Riemannian submersion with totally geodesic fibers.
\end{proof}

\begin{remark}\label{invariant-w}
If we consider on $\mathbb{C}^{(n-k)\times k}$ the probability measure
\[
d\mu:=c_{n,k}\det(I_k+w^*w)^{-n} dw
\]
with normalizing constant $c_{n,k}$, then a direct computation shows that $\mu$ is the symmetric and invariant measure for the diffusion $(w_t)_{t \ge 0}$, i.e. if $f,g$ are smooth and compactly supported functions on $\mathbb{C}^{(n-k)\times k}$, then
\[
\int_{\mathbb{C}^{(n-k)\times k}} (\Delta_{\widehat{G}_{n,k}}f) g\, d\mu=\int_{\mathbb{C}^{(n-k)\times k}} f(\Delta_{\widehat{G}_{n,k}}g)\, d\mu
\]
\end{remark}

\begin{remark}
The complex Grassmannian $G_{n,k}$ is a compact irreducible  symmetric space of rank $k$ and the complex Stiefel manifold is a Riemannian homogeneous space (but is not symmetric). As such, the Brownian motions on $G_{n,k}$ and $V_{n,k}$ and their distributions  and pathwise properties can be studied using representation theory and stochastic differential geometry. The literature on those topics is nowadays quite substantial. We for instance refer to the early  works  by Eugene Dynkin \cite{Dyn61, MR0202206} and Paul $\&$ Marie-Paule Malliavin \cite{MR0359023} or more recent presentations like  \cite{MR3201989} and  the book \cite{MR2731662} (see in particular Chapter 8: Riemannian submersions and Symmetric spaces). In some sense, Theorem \ref{main:s1} provides a more pedestrian approach: We work in a specific choice of coordinates within the algebra of complex matrices and describe the $G_{n,k}$  and $V_{n,k}$ Brownian motions in those coordinates taking advantage of the  additional structure given by the matrix multiplication.
\end{remark}

\subsection{Eigenvalue process} 

In this section, for later use, we collect some properties of the eigenvalues of the process $(J_t)_{t \ge 0}: = (w^*_t w_t)_{t \ge 0}$  where $(w_t)_{t \ge 0}=(X_tZ_t^{-1})_{t \ge 0}$ is a Brownian motion on $\widehat{G}_{n,k}$ as in Theorem \ref{main:s1}. We note that
\[
J_t= w_t^* w_t=(Z_t^{-1})^* X_t^* X_tZ_t^{-1}=(Z_t Z_t^*)^{-1}-I_k
\]
Therefore $(I_k+J)^{-1}=ZZ^{*}$, and the properties of $J$ and its eigenvalues can be deduced from the corresponding properties of the matrix Jacobi process $ZZ^{*}$ after basic algebraic manipulations. In particular, one immediately has the following result.

\begin{lemma}\label{thm-J}

Let $(J_t)_{t\ge0}$ be given as above, then $\mathbb{P}( \inf\{ t>0, \det (J_t)=0\}<+\infty)=0$ and there exists a Brownian motion $(\bB_t)_{t \ge 0}$ in $\C^{k \times k}$ such that:
\begin{equation}\label{eq-dJ}
dJ=\sqrt{I_{k}+J} d\bB^*\sqrt{I_k+J} \sqrt{J}+\sqrt{J}\sqrt{I_k+J}d\bB \sqrt{I_{k}+J}+2(n-k+\tr(J))(I_k+J)dt
\end{equation} 
\end{lemma}

We explicitly note that the proof of $\mathbb{P}( \inf\{ t>0, \det (J_t)=0\}<+\infty)=0$ uses the fact that $k \le n-k$.

\begin{remark}
The symmetric and invariant probability measure of the diffusion process $J$ can be easily obtained. Indeed, let $\mathbf{w}$ be a random variable on $\C^{(n-k)\times k}$ whose law is the probability measure with density $c_{n,k}\det(I_k+w^*w)^{-n} dw$.  From Proposition 1 in \cite{Forrester} one has for every bounded Borel function $g$ and some normalization constant $c'_{n,k}$ that
\begin{align*}
\mathbb{E} ( g (\mathbf{w}^* \mathbf{w}))&=c_{n,k} \int_{\C^{(n-k)\times k}} g (w^* w) \det (I_k +w^* w)^{-n} dw \\
 & =c'_{n,k} \int_{\Hh_k} g (S) \det (I_k +S)^{-n} \det(S)^{n-2k} dS.
\end{align*}
Thus, from Remark \ref{invariant-w}, the probability measure on the cone $\Hh_k$ of positive definite Hermitian matrices with density  $ c'_{n,k} \det (I_k +S)^{-n} \det(S)^{n-2k} dS$ is the invariant and symmetric probability measure for the diffusion process $J$.
\end{remark}

Concerning the eigenvalues of $J$ one obtains after applying  techniques and  results from \cite{ demni2010beta, doumerc2005matrices,Gra99,graczyk2013multidimensional} the following result:

\begin{lemma}\label{thm-lambda}
Let   $\lambda(t)=(\lambda_1(t),\dots,\lambda_k(t))$, $t\ge0$ be the eigenvalue process of the diffusion matrix $J$. Assume that $\lambda_1(0) > \cdots > \lambda_k(0)$, then the process $\lambda(t)$ is non colliding, i.e.
\[
\mathbb{P}\left( \forall \, t \ge 0, \lambda_1(t) > \cdots > \lambda_k(t) \right)=1.
\]
Moreover, we have
\begin{equation}\label{eq-lambda-sde}
d\lambda_i=2(1+\lambda_i)\sqrt{\lambda_i}dB^i+2(1+\lambda_i) \bigg(n-2k+1-(2k-3)\lambda_i+2\lambda_i (1+\lambda_i)\sum_{\ell\not=i} \frac{1}{\lambda_i-\lambda_\ell}\bigg)dt,
\end{equation}
where $(B_t)_{t \ge 0}$ is a Brownian motion in $\mathbb{R}^k$.
\end{lemma}

Under the assumptions of the previous theorem, let us denote $\rho_i=\frac{1-\lambda_i}{1+\lambda_i}$, $i=1,\dots, k$. Note that $\rho$ is the eigenvalue process of $2ZZ^*-I_k$. Then, as an application of It\^o's formula, we have
\begin{align*}
d\rho_i =-2\sqrt{1-\rho_i^2}dB^i-2\bigg((n-2k+(n-2k+2)\rho_i)+2\sum_{\ell\not=i}\frac{1-\rho_i^2}{\rho_\ell-\rho_i}\bigg)dt,
\end{align*}
where $(B_t)_{t \ge 0}$ is the same Brownian motion as in \eqref{eq-lambda-sde}. Therefore, $ \rho$ is a diffusion process  with generator given by

\[
\mathcal{L}_{n,k}=2\sum_{i=1}^k(1-\rho_i^2)\partial_i^2-2\sum_{i=1}^k\bigg(n-2k+(n-2k+2)\rho_i+2\sum_{\ell\not=i} \frac{1-\rho_i^2}{\rho_\ell-\rho_i}\bigg)\partial_i.
\]

If we consider the Vandermonde function
\[
h(\rho)=\prod_{i>j}(\rho_i-\rho_j),
\]
then we have for every smooth function $f$ on $[-1,1]^k$ that
\[
\mathcal{L}_{n,k}f=2 \left( \frac{1}{h}  \mathcal{G}_{n-2k, 0} (h f) +\frac{1}{6} k(k-1) \left( 3n-4k+2 \right)  f \right).
\]
where $\mathcal{G}_{a,b}:=\sum_{i=1}^k (1-\rho^2_i)\partial_i^2-(a-b+(a+b+2)\rho_i)\partial_i $ is a sum of Jacobi diffusion operators on $[-1,1]$.  Therefore $(\rho(t))_{t \ge 0}$ is  a Karlin-McGregor  diffusion associated to a $k$-dimensional Jacobi process with independent components and conditioned by its ground state. Using then the well-known  Karlin-McGregor  formula,  see \cite{interlacing}, one deduces that for every $t>0$, $\rho(t)$ has a density with respect to the Lebesgue measure $dx$ given by 
\[
e^{ \frac{1}{3} k(k-1) \left( 3n-4k+2 \right) t}  \frac{\prod_{i>j}(x_i-x_j)}{\prod_{i>j}(\rho_i(0)-\rho_j(0))} \mathrm{det} \left( p^{n-2k,0}_{t} (\rho_i(0),x_j)\right)_{1\le i, j \le k}  \, \mathbf{1}_{\Delta_k} (x),
\]
where $p^{n-2k,0}_{t}$ is the transition density of a one-dimensional Jacobi diffusion (see \eqref{eq-jacobi-kernel} for a precise formula) and 
\[
\Delta_k=\{ x \in [-1,1]^k, -1 \le x_1 <\cdots <x_k \le 1 \}. 
\]
 Moreover, when $t \to +\infty$, $\rho(t)$ converges in distribution to the invariant probability measure
\begin{align}\label{limit coulomb}
 d\nu=c_{n,k} \prod_{1 \le i < j \le k}(x_i-x_j)^2  \prod_{i=1}^k (1-x_i)^{n-2k}  \, \, \mathbf{1}_{\Delta_k} (x) dx,
\end{align}
where $c_{n,k}$ is again a normalization constant and it might   be explicitly computed using the well-known Selberg integral formula (see \cite{MR2760897}).

\section{Skew-product decomposition of the Brownian motion of the Stiefel fibration}

\subsection{Connection form and horizontal Brownian motion on $V_{n,k}$}

Let us consider the Stiefel fibration
\begin{equation}\label{eq-fibration2}
\mathbf{U}(k) \rightarrow\widehat{V}_{n,k} \rightarrow \widehat{G}_{n,k}
\end{equation}
that was described in Sections \ref{fibration} and \ref{inhomogeneous}. According to this fibration, one can see $\widehat{V}_{n,k}$ as a $\mathbf{U}(k)$-principal bundle over $\widehat{G}_{n,k}$. The next lemma gives a formula for the connection form of this bundle.

\begin{lemma}\label{connection form}
Consider on $\widehat{V}_{n,k}$
the $\mathfrak{u}(k)$-valued one form
\begin{align}\label{eq-contact-x}
\omega:=\frac{1}{2} \left( (X^* \,  Z^*)d\begin{pmatrix}  X \\ Z  \end{pmatrix}-d(X^* \, Z^*)\begin{pmatrix}  X \\ Z  \end{pmatrix}\right) =\frac{1}{2} \left(X^*dX-dX^*X +Z^*dZ-dZ^*Z\right).
\end{align}
Then, $\omega$ is the connection form of the bundle $\mathbf{U}(k) \rightarrow\widehat{V}_{n,k} \rightarrow \widehat{G}_{n,k}$.
\end{lemma}

\begin{proof}
We first observe that if $v=\begin{pmatrix} X  \\ Z \end{pmatrix} \in V_{n,k}$, then the tangent space to $V_{n,k}$ at $v$ is given by
\[
T_v V_{n,k}= \left\{ \begin{pmatrix} A  \\ B \end{pmatrix} \in \C^{n \times k}, A^* X+X^* A+B^*Z +Z^*B=0  \right\}.
\]
Then, if $\theta \in \mathfrak{u}_k$, one easily computes that the generator of the one-parameter group $\{ q \to q e^{t \theta}\}_{t \in \mathbb{R}} $ is given by the vector field on $V_{n,k}$ whose value at $v$ is $\begin{pmatrix} X \theta  \\ Z \theta \end{pmatrix}$. Applying $\omega$ to this vector field yields $\theta$.   To show that $\omega$ is the connection form it remains therefore to prove that the kernel of $\omega$ is the horizontal space of the Riemannian submersion $\begin{pmatrix} X  \\ Z \end{pmatrix} \to XZ^{-1}$. This horizontal space at $v$, say $\mathcal{H}_v$, is the orthogonal complement of the vertical space at $v$, which is the subspace $\mathcal{V}_v$ of $T_v V_{n,k}$ tangent to the fiber of the submersion. The previous argument shows that
\[
\mathcal{V}_v = \left\{ \begin{pmatrix} X\theta  \\ Z\theta \end{pmatrix}, \theta \in \mathfrak{u}(k) \right\}.
\]
Therefore we have
\[
\mathcal{H}_v = \left\{ \begin{pmatrix} A  \\ B  \end{pmatrix}\in T_v V_{n,k}, \, \forall \,   \theta \in \mathfrak{u}(k), \, \mathrm{tr} \left( A^* X \theta +B^* Z \theta  \right)=0 \right\}.
\]
We deduce from this that
\[
\mathcal{H}_v = \left\{ \begin{pmatrix} A  \\ B  \end{pmatrix}\in T_v V_{n,k}, \,  A^* X  +B^* Z =X^* A +Z^*B \right\},
\]
from which it is clear that $\omega_{\mid \mathcal{H}}=0$.
\end{proof}

Our next goal is to describe the horizontal lift to $\widehat{V}_{n,k}$ of a Brownian motion on $\widehat{G}_{n,k}$. We still denote by $p:\widehat{V}_{n,k} \to \widehat{G}_{n,k}$ the Riemannian submersion. A continuous semimartingale $(M_t)_{t \ge 0}$ on $\widehat{V}_{n,k}$ is called \textit{horizontal} if for every $t\ge 0$,   $\int_{M[0,t]} \omega =0$, where $\int_{M[0,t]} \omega$ denotes   the Stratonovich line integral of $\omega$ along the paths of $M$. 
 If $(N_t)_{t \ge 0}$ is a continuous semimartingale on $\widehat{G}_{n,k}$ with $N_0 \in \widehat{G}_{n,k}$, then if $\widetilde{N}_0 \in\widehat{V}_{n,k}$ is such that $\ p(\widetilde{N}_0)=N_0$,  there exists a unique horizontal continuous semimartingale $(\widetilde{N}_t)_{t \ge 0}$ on $\widehat{V}_{n,k}$ such that $p(\widetilde{N}_t)=N_t$ for every $t \ge 0$. The semimartingale $(\widetilde{N}_t)_{t \ge 0}$ is then called the horizontal lift at $\widetilde{N}_0$ of $(N_t)_{t \ge 0}$ to $\widehat{V}_{n,k}$. We refer to \cite{MR3969194} or \cite{MR2731662} for a more general description of the horizontal lift of a semimartingale in the context of foliations.

 We then consider on $\widehat{G}_{n,k}$ the $\mathfrak{u}(k)$-valued one-form $\eta$ given  by
 \begin{align}\label{defeta}
 \eta:=& \frac12 \left( (I_k+w^* w)^{-1/2} (dw^* \, w-w^*dw)(I_k+w^* w)^{-1/2}  \right. \\
  & \left. - (I_k+w^* w)^{-1/2}\, d(I_k+w^* w)^{1/2}+d(I_k+w^* w)^{1/2} \, (I_k+w^* w)^{-1/2} \right). \notag
 \end{align}

\begin{theorem}\label{lift}
Let $(w_t)_{t \ge 0}$ be a Brownian motion on $\widehat{G}_{n,k}$ started at $w_0 \in \widehat{G}_{n,k}$ as in Theorem \ref{main:s1} and $\mathfrak{a}_t= \int_{w[0,t]} \eta$. Let $\begin{pmatrix} X_0  \\ Z_0 \end{pmatrix} \in\widehat{V}_{n,k}$ be such that $X_0Z_0^{-1}=w_0$. The process
\[
\widetilde{w}_t:=\begin{pmatrix} w_t  \\  I_k  \end{pmatrix}(I_k+w_t^*w_t)^{-1/2}\Theta_t
\]
is the horizontal lift at $\begin{pmatrix} X_0  \\ Z_0 \end{pmatrix}$ of $(w_t)_{t \ge 0}$ to $\widehat{V}_{n,k}$, where $(\Theta_t)_{t \ge 0}$ is the $\mathbf{U}(k)$-valued  solution of the Stratonovich stochastic differential equation
\begin{align*}
\begin{cases}
d\Theta_t = \circ d\mathfrak a_t \, \Theta_t \\
\Theta_0=(Z_0 Z^*_0)^{-1/2}Z_0.
\end{cases}
\end{align*}
\end{theorem}

\begin{proof}
As before we denote by $p$ the submersion $\begin{pmatrix} X  \\ Z \end{pmatrix} \to XZ^{-1}$. It is easy to check that for every $t \ge 0$, $p( \widetilde{w}_t)=w_t$ and that $\widetilde{w}_0=\begin{pmatrix} X_0  \\ Z_0 \end{pmatrix}$. It is therefore enough to prove that $\widetilde{w}$ is a horizontal semimartingale, i.e. that $\int_{\widetilde{w}[0,t]} \omega=0$. Denote
\[
X_t= w_t(I_k+w_t^*w_t)^{-1/2}\Theta_t, \, \, Z_t=(I_k+w_t^*w_t)^{-1/2}\Theta_t
\]
A long, but routine, computation shows that
\begin{align*}
&\frac{1}{2} \left(X^*\circ dX-\circ dX^*X +Z^*\circ dZ-\circ dZ^*Z\right)\\
=&-\frac{1}{2} \left( \circ d\Theta^*\Theta-\Theta^*\circ d\Theta+\Theta^*\bigg(\circ d(I_k+J)^{-1/2}\,(I_k+J)^{1/2}-(I_k+J)^{1/2}\circ d(I_k+J)^{-1/2}\bigg)\Theta \right. \\
&+\left. \Theta^*(I_k+J)^{-1/2}(\circ dw^*w-w^*\circ dw)(I_k+J)^{-1/2}\Theta \right).
\end{align*}
where $J =w^* w$. Since $\circ d\Theta^*=\circ d\Theta^{-1}=-\Theta^{-1}\circ d\Theta\, \Theta^{-1}$ and $\circ d\Theta = \circ d\mathfrak a \, \Theta $ with

\begin{align*}
\circ d \mathfrak a=& \frac12(I_k+J)^{-1/2} (\circ dw^* \, w-w^*\circ dw)(I_k+J)^{-1/2} \\
 & -\frac{1}{2} \left( (I_k+J)^{-1/2}\, \circ d(I_k+J)^{1/2}-\circ d(I_k+J)^{1/2} \, (I_k+J)^{-1/2}\right)
 \end{align*}
 we conclude that 
 \[
 \frac{1}{2} \left(X^*\circ dX-\circ dX^*X +Z^*\circ dZ-\circ dZ^*Z\right)=0
 \]
 and thus $\int_{\widetilde{w}[0,t]} \omega=0$.
\end{proof}
\subsection{Skew-product decomposition of the Stiefel Brownian motion}

We now turn to the description of the Brownian motion on $\widehat{V}_{n,k}$ as a skew-product.

\begin{theorem}\label{skew}
Let $(w_t)_{t \ge 0}$ be a Brownian motion on $\widehat{G}_{n,k}$ started at $w_0=X_0Z_0^{-1} \in \widehat{G}_{n,k}$ as in Theorem \ref{main:s1} and let $(\Omega_t)_{t \ge 0}$ be a Brownian motion on the unitary group $\mathbf{U}(k)$ independent from $(w_t)_{t \ge 0}$.  Let $(\Theta_t)_{t \ge 0}$ be the $\mathbf{U}(k)$-valued  solution of the Stratonovich stochastic differential equation
\begin{align*}
\begin{cases}
d\Theta_t = \circ d\mathfrak a_t \, \Theta_t \\
\Theta_0=(Z_0 Z^*_0)^{-1/2}Z_0,
\end{cases}
\end{align*}
where  $\mathfrak{a}_t= \int_{w[0,t]} \eta$.
The process $$\begin{pmatrix} w_t  \\  I_k  \end{pmatrix}(I_k+w_t^*w_t)^{-1/2}\Theta_t \,  \Omega_t$$
is a Brownian motion on $\widehat{V}_{n,k}$ started at $\begin{pmatrix} X_0  \\ Z_0 \end{pmatrix}$.
\end{theorem}

\begin{proof}
We denote by $\Delta_{\mathcal{H}}$ the horizontal Laplacian and by $\Delta_{\mathcal{V}}$ the vertical Laplacian of the Stiefel fibration; see \cite{BIHP} for the definitions of horizontal and vertical Laplacians. Since the submersion $\widehat{V}_{n,k} \to \widehat{G}_{n,k}$ is totally geodesic, the operators $\Delta_{\mathcal{H}}$ and $\Delta_{\mathcal{V}}$ commute (see \cite{BIHP}). We note that the Laplace-Beltrami operator of $\widehat{V}_{n,k}$ is given by $\Delta_{\widehat{V}_{n,k}} = \Delta_{\mathcal{H}}+\Delta_{\mathcal{V}}$ and that the horizontal lift of the Brownian motion on $\widehat G_{n,k}$ is a diffusion with generator $\frac{1}{2} \Delta_{\mathcal{H}}$, see \cite{MR3969194}. The fibers of the submersion $\widehat{V}_{n,k} \to \widehat G_{n,k}$ are isometric to $\mathbf{U}(k)$, thus if $f$ is a bounded Borel function on $\widehat{V}_{n,k}$, one has
\[
e^{\frac{1}{2} t \Delta_{\mathcal{V}}}f \begin{pmatrix} X  \\ Z \end{pmatrix}=  \mathbb{E} \left ( f \begin{pmatrix} X \Omega_t  \\ Z \Omega_t \end{pmatrix}   \right).
\]
Since $e^{\frac{1}{2} t \Delta_{\mathcal{V}}}e^{\frac{1}{2} t \Delta_{\mathcal{H}}}= e^{\frac{1}{2} t \Delta_{\widehat{V}_{n,k} }}$, we conclude from Theorem \ref{lift}.
 \end{proof}

\section{Limit theorems}

Throughout the section, let $(w_t)_{t \ge 0}=(X_tZ_t^{-1})_{t \ge 0}$ be a Brownian motion on $\widehat{G}_{n,k}$ where $\begin{pmatrix} X_t  \\ Z_t \end{pmatrix}_{t \ge 0}$ is a Brownian motion on $\widehat{V}_{n,k}$ . Our goal is to prove Theorem \ref{winding intro}. Without loss of generality we will assume throughout the section that the eigenvalues of $Z_0^*Z_0$ are distinct; Even  if the eigenvalues of the complex Jacobi process $Z_t^*Z_t$ are not distinct for $t=0$, they will be distinct for any $t>0$, see \cite{MR3842169,demni2010beta}, and thus from the Markov property, the limit   Theorem \ref{winding intro} still holds.

\subsection{Main limit theorem}

We first  give a limit theorem for the process $\left(\int_0^t \tr \left( w^*_s w_s) ds \right) \right)_{ t \ge 0}$ that shall be used in the next subsections. The method we use, a Girsanov transform, takes its root in the paper by M. Yor \cite{Yor1} and was further developed in the situation of matrix Wishart diffusions in \cite{doumerc1} and in the situation of the real Jacobi matrix processes in Section 9.4.2 of the thesis \cite{doumerc2005matrices}.  Our result is the following:

\begin{theorem}\label{limit_oi}
Let $(J_t)_{t \ge 0}=(w_t^* w_t)_{t \ge 0}$. The following convergence holds in distribution when $t \to +\infty$
\[
\frac{1}{t^2} \int_0^t  \tr(J)ds \to X,
\]
where $X$ is a random variable on $[0,+\infty)$ with density $\frac{k(n-k)}{\sqrt{2\pi}  x^{3/2} }e^{-\frac{k^2(n-k)^2}{2x}}$ (therefore $X$ is the inverse of a gamma distributed random variable).
\end{theorem}

The proof is rather long and will be decomposed in several steps. We first recall that from Lemma \ref{thm-J}, there exists a Brownian motion $(\bB_t)_{t \ge 0}$ in $\C^{k \times k}$ such that:
\begin{equation}\label{eq-dJ}
dJ=\sqrt{I_{k}+J} d\bB^*\sqrt{I_k+J} \sqrt{J}+\sqrt{J}\sqrt{I_k+J}d\bB \sqrt{I_{k}+J}+2(n-k+\tr(J))(I_k+J)dt
\end{equation}

\begin{lemma}\label{thm-J2}
We have 
\begin{equation}\label{eq-det-I+J}
d(\det(I_k+J))=\det(I_k+J)\tr\left( \sqrt{J}(d\bB+d\bB^*) \right)+2\det(I_k+J)\bigg(k(n-k)+ \tr(J) \bigg)dt,
\end{equation}
and therefore 
\begin{equation}\label{eq-logdet-I+J}
d(\log\det(I_k+J))=\tr\left(\sqrt{J}(d\bB+d\bB^*) \right)+2k(n-k)dt.
\end{equation}
\end{lemma}
\begin{proof}
By It\^o's formula we have
\[
d(\det(I_k+J))=\sum_{i,j=1}^k\frac{\partial \det(I_k+J)}{\partial J_{ij}}dJ_{ij}+\frac12 \sum_{i,j, i', j'=1}^k\frac{\partial^2 \det(I_k+J)}{\partial J_{ij}\partial J_{i'j'}}dJ_{ij}dJ_{i'j'}.
\]
First, we know that
\[
\frac{\partial \det(J)}{\partial J_{ij}}=
\frac{\partial \sum_{\ell=1}^k J_{i\ell}\tilde{J}_{i\ell}}{\partial J_{ij}}=\tilde{J}_{ij}
\]
where $\tilde{J}=\det (J)(J^T)^{-1}$ is the cofactor of $J$. Hence the first order term  writes $\det(I_k+J)\tr ((I_k+J)^{-1}dJ)$. Next, we will use the following formula to compute the cross second order derivatives:
\[
\frac{\partial^2 \det(J)}{\partial x\partial y}=(\det (J))\bigg( \tr \left(J^{-1}\frac{\partial^2 J}{\partial x\partial y} \right)
+\tr \left(J^{-1}\frac{\partial J}{\partial x} \right)\tr \left(J^{-1}\frac{\partial J}{\partial y}  \right)
-\tr\left(J^{-1}\frac{\partial J}{\partial x} J^{-1}\frac{\partial J}{\partial y} \right)
\bigg).
\]
Since $\frac{\partial J}{\partial J_{ij}}=E_{ij}$, clearly $\frac{\partial^2 J}{\partial J_{ij}\partial J_{i'j'}}=0$. We also have $J^{-1}\frac{\partial J}{\partial J_{ij}}=\sum_{\ell}(J^{-1})_{\ell i}E_{\ell j} $ and $\tr \left( J^{-1}\frac{\partial J}{\partial J_{ij}}\right)=(J^{-1})_{ji}$. Hence 
\[
\frac{\partial^2 \det(J)}{\partial J_{ij}\partial J_{i'j'}}=(\det (J))\bigg( 
(J^{-1})_{ji}(J^{-1})_{j'i'}
-(J^{-1})_{j'i}(J^{-1})_{ji'}
\bigg),
\]
and
\[
\frac{\partial^2 \det(I_k+J)}{\partial J_{ij}\partial J_{i'j'}}=(\det (I_k+J))\bigg( 
((I_k+J)^{-1})_{ji}((I_k+J)^{-1})_{j'i'}
-((I_k+J)^{-1})_{j'i}((I_k+J)^{-1})_{ji'}
\bigg).
\]
Moreover, from \eqref{eq-dJ} we know that
\[
dJ_{ij}dJ_{i'j'}=2dt\bigg((J+J^2)_{i'j}(I_k+J)_{ij'}+(J+J^2)_{ij'}(I_k+J)_{i'j} \bigg)
\]
Hence we have
\begin{align*}
&d(\det(I_k+J))=\det(I_k+J)\tr ((I_k+J)^{-1}dJ)\\
&+ \sum_{i,j, i', j'=1}^k\det (I_k+J)\bigg( 
((I_k+J)^{-1})_{ji}((I_k+J)^{-1})_{j'i'}
-((I_k+J)^{-1})_{j'i}((I_k+J)^{-1})_{ji'}
\bigg)\\
&\quad\quad\cdot\bigg((J+J^2)_{i'j}(I_k+J)_{ij'}+(J+J^2)_{ij'}(I_k+J)_{i'j} \bigg)dt\\
&=\det(I_k+J)\tr ((I_k+J)^{-1}dJ)
- 2(k-1) \det(I_k+J)\tr(J)dt.
\end{align*}
From \eqref{eq-dJ} we know
\[
\tr ((I_k+J)^{-1}dJ)=\tr\left(\sqrt{J}(d\bB+d\bB^*) \right)+ 2k (n-k+\tr (J))dt.
\]
Hence
\begin{align*}
&d(\det(I_k+J))=\det(I_k+J)\,\tr\left(\sqrt{J}(d\bB+d\bB^*) \right)+ 2\det(I_k+J)(k(n-k)+ \tr(J))dt.
\end{align*}
As a direct consequence of $d\langle \det(I_k+J),\det(I_k+J)\rangle=4\,\det(I_k+J)^2\tr(J)dt$, we obtain \eqref{eq-logdet-I+J} using It\^o's formula.
\end{proof}

\begin{lemma}\label{martingale g}
For every $\alpha \ge 0$ the process
\[
M_t^{\alpha}=e^{2k\alpha(n-k)t}  \left(\frac{\det(I_k+J_0)}{\det(I_k+J_t)}\right)^\alpha \exp\left(-2\alpha^2 \int_0^t \tr(J)ds \right)
\]
is a martingale.
\end{lemma}

\begin{proof}
Consider the exponential local martingale
\[
M_t^\alpha:=\exp\bigg(-\alpha\int_0^t \tr (\sqrt{J} (d\bB+ d\bB^*))-2\alpha^2\int_0^t \tr(J)ds \bigg),
\] 
where $\bB$ is the Brownian motion as given in Theorem \ref{thm-J2}. From Lemma \ref{thm-J2}, we have 
\[
\left(\frac{\det(I_k+J_t)}{\det(I_k+J_0)}\right)^\alpha=\exp\bigg(\alpha\left(\int_0^t\tr\left(\sqrt{J}(d\bB+d\bB^*) \right)+ 2k(n-k) ds \right)\bigg),
\]
and thus
\[
M_t^\alpha=e^{2k\alpha(n-k)t}  \left(\frac{\det(I_k+J_0)}{\det(I_k+J_t)}\right)^\alpha \exp\left(-2\alpha^2 \int_0^t \tr(J)ds \right).
\]
From this expression, it is clear that there exists a constant $C>0$ such that we almost surely have $|M_t^\alpha| \le Ce^{2k\alpha(n-k)t} $ and thus the process $(M_t^\alpha)_{t \ge 0}$ is a martingale.
%
%
\end{proof}

In the next lemma, we provide a formula for the Laplace transform of the functional $\int_0^t \tr(J)ds$ using a Girsanov transform. In this computation,  the transition kernel of one-dimensional Jacobi diffusions naturally appears. We recall the formula for this transition kernel. If we denote by $p^{a,b}_t(x,y)$ the transition density, with respect to the Lebesgue measure, of the diffusion with generator 
\[
2 (1-x^2)\frac{\partial^2}{\partial x^2}-2\left( (a+b+2)x +a -b \right)\frac{\partial}{\partial x}
\]
and initiated from $x \in (-1,1)$, then we have

\begin{align}\label{eq-jacobi-kernel}
 & p^{a,b}_t(x,y)=(1+y)^{b}(1-y)^{a} \sum_{m=0}^{+\infty} c_{m,a, b} e^{-2m(m+a+b+1)t}P_m^{a,b}(x)P_m^{a,b}(y),
\end{align}
where $c_{m,a, b}= \frac{2m+a+b+1}{2^{a+b+1}}\frac{\Gamma(m+a+b+1)\Gamma(m+1)}{\Gamma(m+a+1)\Gamma(m+b+1)}$ and where  $P_m^{a,b}(x)$, $m\in \mathbb{Z}_{\ge0}$ are the  Jacobi polynomials given by
\[
P_m^{a,b}(x)=\frac{(-1)^m}{2^mm!(1-x)^{a}(1+x)^b}\frac{d^m}{dx^m}((1-x)^{a+m}(1+x)^{b+m}).
\]

\begin{lemma}\label{kernel J}
For every $\alpha \ge 0$ and $t>0$
\begin{align*}
 & \E\left( e^{-2\alpha^2 \int_0^t \tr(J)ds}\right) \\
 =& C e^{ \left(\frac{1}{3} k(k-1) \left( 3n-4k+6\alpha+2 \right) -2k(n-k)\alpha\right) t}  \int_{\Delta_k} \mathrm{det} \left( \frac{p^{n-2k,2\alpha}_{t} \left(\frac{1-\lambda_i(0)}{1+\lambda_i(0)},x_j\right)}{(1+x_j)^{\alpha}}\right)_{ i, j }  \prod_{i>j}(x_i-x_j) \, \, dx.
\end{align*}
where 
\[
C=\prod_\ell \frac{( 1+\lambda_\ell(0))^\alpha}{2^\alpha} \prod_{i>j}  \frac{(1+\lambda_i(0))(1+\lambda_j(0))}{2(\lambda_j(0)-\lambda_i(0))}
\]
 is the normalization constant, $\lambda_1(0),\cdots,\lambda_k(0)$ are the ordered eigenvalues of $J_0$,  $p^{n-2k,2\alpha}_{t}$ is given by the formula \eqref{eq-jacobi-kernel}  and 
\[
\Delta_k=\{ x \in [-1,1]^k, -1 \le x_1 <\cdots <x_k \le 1 \}. 
\]

\end{lemma}

\begin{proof}
Let $\alpha \ge 0$ and consider the probability measure $P^{\alpha}$ defined by
\[
P^{\alpha}|_{\mathcal{F}_t}=M^{\alpha}_t\cdot P|_{\mathcal{F}_t}.
\] 
We first note that
\begin{align}\label{e-alpha}
\E\left( e^{-2\alpha^2 \int_0^t \tr(J)ds}\right)=e^{-2k(n-k)\alpha t}\, \E^{\alpha} \left[ \left(\frac{\det(I_k+J_t)}{\det(I_k+J_0)}\right)^\alpha  \right].
\end{align}
From Girsanov theorem, the process
\[
\beta_t=\bB_t+2 \alpha\int_0^t\sqrt{J} ds
\]
is under $P^\alpha$ a $k \times k$-matrix-valued Brownian motion and we have
\[
dJ=\sqrt{I_{k}+J} d\beta^*\sqrt{I_k+J} \sqrt{J}+\sqrt{J}\sqrt{I_k+J}d\beta \sqrt{I_{k}+J}+2\left(n-k-2\alpha J+\tr(J)\right)(I_k+J)dt.
\]
We now denote by $\lambda(t)= (\lambda_i(t))_{1\le i\le k}$ the eigenvalues of $J_t$, $t\ge0$. From the previous equation satisfied by $J$ we deduce that there exists a Brownian motion $(B_t)_{t \ge 0}$ in $\mathbb{R}^k$ for the probability measure $P^\alpha$  such that
\begin{equation}\label{eq-lambda-sde-Pl}
d\lambda_i=2(1+\lambda_i)\sqrt{\lambda_i}dB^i+2(1+\lambda_i) \bigg(n-2k+1-(2k+2\alpha-3)\lambda_i+2\lambda_i (1+\lambda_i)\sum_{\ell\not=i} \frac{1}{\lambda_i-\lambda_\ell}\bigg)dt.
\end{equation}
Let us denote  $\rho_i=\frac{1-\lambda_i}{1+\lambda_i}$, $i=1,\dots, k$. Then, using It\^o's formula and the previous equation, we have
\begin{align*}
d\rho_i =-2\sqrt{1-\rho_i^2}dB^i-2\left( \left(n-2k-2\alpha+(n-2k+6\alpha+2)\rho_i\right)+2\sum_{\ell\not=i}\frac{1-\rho_i^2}{\rho_\ell-\rho_i}\right)dt.
\end{align*}
Using then the formula for the density of non-colliding Jacobi processes, see \cite{interlacing}, we deduce that the process $(\rho(t))_{t \ge 0}$ has, under $P^\alpha$, a density with respect to the Lebesgue measure $dx$ given by 
\[
e^{ \frac{1}{3} k(k-1) \left( 3n-4k+6\alpha+2 \right) t}  \frac{\prod_{i>j}(x_i-x_j)}{\prod_{i>j}(\rho_i(0)-\rho_j(0))} \mathrm{det} \left( p^{n-2k,2\alpha}_{t} (\rho_i(0),x_j)\right)_{1\le i, j \le k}  \, \mathbf{1}_{\Delta_k} (x).
\]
We conclude then with \eqref{e-alpha}.
\end{proof}

We are now ready for the proof of Theorem \ref{limit_oi}.

\begin{proof}[Proof of Theorem \ref{limit_oi}]

 From Lemma \ref{kernel J}, we have for every $\alpha \ge 0$ and $t>0$
\begin{align}\label{laplace-asymp}
 & \E\left( e^{-2\alpha^2 \int_0^t \tr(J)ds}\right) \\
 =& C e^{ \left(\frac{1}{3} k(k-1) \left( 3n-4k+6\alpha+2 \right) -2k(n-k)\alpha\right) t}  \int_{\Delta_k} \mathrm{det} \left( \frac{p^{n-2k,2\alpha}_{t} \left(\frac{1-\lambda_i(0)}{1+\lambda_i(0)},x_j\right)}{(1+x_j)^{\alpha}}\right)_{ i, j }  \prod_{i>j}(x_i-x_j) \, \, dx. \notag
\end{align}

In order to analyze the large time behavior of this Laplace transform  we we will use the formula \eqref{eq-jacobi-kernel}.  We can write
\begin{align*}
 & p^{n-2k,2\alpha}_t(x,y)=(1+y)^{2\alpha}(1-y)^{n-2k}\sum_{m=0}^{+\infty} c_{m,n-2k+ 2\alpha} e^{-2m(m+n-2k+2\alpha+1)t}P_m^{n-2k,2\alpha}(x)P_m^{n-2k,2\alpha}(y),
\end{align*}
Similarly to \cite{demni2010beta}, or Section 3.9.1 in \cite{interlacing}, denoting as before $\rho_i(0)=\frac{1-\lambda_i(0)}{1+\lambda_i(0)}$ we now compute 
\begin{align*}
 & \mathrm{det} \left( p^{n-2k,2\alpha}_{t} (\rho_i(0),x_j)\right)_{1\le i, j \le k} \\
 =& \sum_{\sigma \in \mathfrak{S}_k} \mathrm{sgn}(\sigma) \prod_{i=1}^k  p^{n-2k,2\alpha}_{t}(\rho_{\sigma(i)} (0),x_i) \\
 =& \sum_{\sigma \in \mathfrak{S}_k} \mathrm{sgn}(\sigma) \prod_{i=1}^k \bigg[ (1-x_i)^{n-2k} (1+x_i)^{2\alpha} \sum_{m=0}^{+\infty}  c_{m,n-2k+ 2\alpha} e^{-2m(m+n-2k+2\alpha+1)t}P_m^{n-2k,2\alpha}(\rho_{\sigma(i)} (0))P_m^{n-2k,2\alpha}(x
 _i)\bigg]\\
 =& V_\alpha(x) \sum_{\sigma \in \mathfrak{S}_k} \mathrm{sgn}(\sigma)  \sum_{m_1,\cdots,m_k=0}^{+\infty}  \prod_{i=1}^k c_{m_i,n-2k+ 2\alpha} e^{-2m_i(m_i+n-2k+2\alpha+1)t}P_{m_i}^{n-2k,2\alpha}(\rho_{\sigma(i)} (0))P_{m_i}^{n-2k,2\alpha}(x_i)
\end{align*}
where $V_\alpha (x)=\prod_{i=1}^k (1-x_i)^{n-2k}(1+x_i)^{2\alpha}$. We can now write
\begin{align*}
 &  \sum_{\sigma \in \mathfrak{S}_k} \mathrm{sgn}(\sigma)  \sum_{m_1,\cdots,m_k=0}^{+\infty}  \prod_{i=1}^k c_{m_i,n-2k+ 2\alpha} e^{-2m_i(m_i+n-2k+2\alpha+1)t}P_{m_i}^{n-2k,2\alpha}(\rho_{\sigma(i)} (0))P_{m_i}^{n-2k,2\alpha}(x_i) \\
 =&  \sum_{m_1,\cdots,m_k=0}^{+\infty} \left(\prod_{i=1}^k c_{m_i,n-2k+ 2\alpha} e^{-2m_i(m_i+n-2k+2\alpha+1)t} P_{m_i}^{n-2k,2\alpha}(x_i) \right)\sum_{\sigma \in \mathfrak{S}_k} \mathrm{sgn}(\sigma) \prod_{i=1}^k  P_{m_i}^{n-2k,2\alpha}(\rho_{\sigma(i)} (0)) \\
 =&   \sum_{m_1,\cdots,m_k=0}^{+\infty} \left(\prod_{i=1}^k c_{m_i,n-2k+ 2\alpha} e^{-2m_i(m_i+n-2k+2\alpha+1)t} P_{m_i}^{n-2k,2\alpha}(x_i) \right)  \mathrm{det} \left( P_{m_i}^{n-2k,2\alpha}(\rho_{j} (0))\right)_{1 \le i,j \le k}.
\end{align*} 
By skew-symmetrization, we can rewrite the previous sum as
\begin{align*}
 \sum_{m_1 < \cdots <m_k} \left(\prod_{i=1}^k c_{m_i,n-2k+ 2\alpha} e^{-2m_i(m_i+n-2k+2\alpha+1)t}  \right)\mathrm{det} \left( P_{m_i}^{n-2k,2\alpha}(x_j)\right)_{1 \le i,j \le k}  \mathrm{det} \left( P_{m_i}^{n-2k,2\alpha}(\rho_{j} (0))\right)_{1 \le i,j \le k}.
\end{align*}
Let us note that when $t \to +\infty$, the term of leading order in this sum corresponds to  $(m_1,\cdots,m_k)=(0,1,\cdots,k-1)$ and is given by
\begin{align}\label{skew-fu}
\left(\prod_{i=1}^k c_{i-1,n-2k+ 2\alpha}   \right) e^{- \frac{1}{3} k(k-1) \left( 3n-4k+6\alpha+2 \right) t}   \mathrm{det} \left( P_{i-1}^{n-2k,2\alpha}(x_j)\right)_{1 \le i,j \le k}  \mathrm{det} \left( P_{i-1}^{n-2k,2\alpha}(\rho_{j} (0))\right)_{1 \le i,j \le k}
\end{align}

On the other hand, from \eqref{laplace-asymp} one has for every $\lambda \ge 0$ and $t>0$ that
\begin{align*}
 & \E\left( e^{-\frac{\lambda}{t^2} \int_0^t \tr(J)ds}\right) \\
 =& C e^{ \left(\frac{1}{3} k(k-1) \left( 3n-4k+6\frac{\sqrt{\lambda}}{{\sqrt{2} t}}+2 \right) -2k(n-k)\frac{\sqrt{\lambda}}{{\sqrt{2} t}}\right) t}  \int_{\Delta_k} \mathrm{det} \left( \frac{p^{n-2k,2\frac{\sqrt{\lambda}}{{\sqrt{2} t}}}_{t} \left(\rho_i(0),x_j\right)}{(1+x_j)^{\frac{\sqrt{\lambda}}{{\sqrt{2} t}}}}\right)_{ i, j }  \prod_{i>j}(x_i-x_j) \, \, dx.
\end{align*}
Using \eqref{skew-fu}, one then deduces that for every $\lambda \ge 0$,
\begin{equation}\label{eq-limit-Laplace}
\lim_{t \to +\infty} \E\left( e^{-\frac{\lambda}{t^2} \int_0^t \tr(J)ds}\right) = \tilde{C} e^{  -k(n-k) \sqrt{2\lambda} } ,
\end{equation}
where $\tilde{C}$ is a constant depending only on $n,k,\rho_i(0)$. For $\lambda=0$, $\E\left( e^{-\frac{\lambda}{t^2} \int_0^t \tr(J)ds}\right)=1$ and therefore $\tilde{C}=1$.
One now concludes using an inverse Laplace transform that the following convergence takes place in distribution when $t \to +\infty$
\[
\frac{1}{t^2} \int_0^t  \tr(J)ds \to X,
\]
where $X$ is a random variable on $[0,+\infty)$ with density $\frac{k(n-k)}{\sqrt{2\pi}  x^{3/2} }e^{-\frac{k^2(n-k)^2}{2x}}$. We incidentally note that $X$ is therefore distributed as the hitting time of $k(n-k)$ by a one-dimensional Brownian motion, even though this does not seem to be readily explainable. 
%
%
%
%
%
\end{proof}

\subsection{Asymptotics of a generalized stochastic area}

By the definition of the one-form $\eta$ in \eqref{defeta}, we note that
\begin{align}\label{eq-tr-eta}
\int_{w[0,t]} \mathrm{tr} (\eta) & =\frac12 \mathrm{tr} \left[ \int_0^t   (I_k+J)^{-1/2} (\circ dw^* \, w-w^*\circ dw)(I_k+J)^{-1/2}\right] \notag\\
 & =\frac12 \mathrm{tr} \left[ \int_0^t   (I_k+J)^{-1/2} ( dw^* \, w-w^* dw)(I_k+J)^{-1/2}\right]
\end{align}
where as before $J=w^* w$. From simple computations one can verify that
\[
\tr(d\eta)=\partial \overline{\partial}\log \det(I_k+w^*w),
\]
which implies that $i\tr(d\eta)$ is the K\"ahler form on $\widehat{G}_{n,k}$. Therefore $i\int_{w[0,t]} \mathrm{tr} (\eta)$ can be, in some sense, considered as a generalized stochastic area process on $\widehat{G}_{n,k}$; we refer to \cite{baudoin2017stochastic} for further explanation on the terminology of generalized stochastic area. In the proposition below we deduce large time limit distributions of such functionals.

\begin{proposition}\label{thm-main-3}

The following convergence holds in distribution when $t \to +\infty$
\[
\frac{1}{it} \int_{w[0,t]} \mathrm{tr} (\eta) \to \mathcal{C}_{k(n-k)},
\]
where $\mathcal{C}_{k(n-k)}$ is a Cauchy distribution of parameter $k(n-k)$.
\end{proposition}

\begin{proof}
Using \eqref{w dynamics}, similarly to the proof of \eqref{eq-dJ} one can check that
\begin{align*}
dw^*w-w^*dw=\sqrt{I_{k}+J} d\mathbf{B}^*\sqrt{I_k+J}\sqrt{J}-\sqrt{J}\sqrt{I_k+J}d\mathbf{B}\sqrt{I_{k}+J}
\end{align*}
where $(\bB_t)_{t\ge0}$ is a $k\times k$-matrix-valued Brownian motion.
Therefore,
\[
 (I_k+J)^{-1/2} ( dw^* \, w-w^* dw)(I_k+J)^{-1/2}= d\bB^*\sqrt{J}- \sqrt{J}d\bB
\]
Consider then the diagonalization of $J=V\Lambda V^*$, where $V\in U(k)$ and $\Lambda=\mathrm{diag} \{\lambda_1,\dots, \lambda_k\}$. We obtain
\begin{align*}
d\bB^*\sqrt{J}- \sqrt{J}d\bB=V( V^{-1}d\bB^* V \sqrt{\Lambda}- \sqrt{\Lambda}V^{-1} d\bB V) V^{-1}.
\end{align*}
Therefore from \eqref{eq-tr-eta}, we have in distribution that
\[
\int_{w[0,t]} \mathrm{tr} (\eta)=  i \mathcal{B}_{\int_0^t \mathrm{tr}(J) ds} 
\]
where $\mathcal{B}$ is a one-dimensional Brownian motion independent from the process $\mathrm{tr}(J)$. Therefore, for every $\lambda>0$,
\[
\mathbb{E} \left( e^{-\lambda\frac{1}{i} \int_{w[0,t]} \mathrm{tr} (\eta)}  \right)=\mathbb{E} \left( e^{-\lambda  \mathcal{B}_{\int_0^t \mathrm{tr}(J) ds}} \right)=\mathbb{E} \left( e^{-\frac{\lambda^2}{2} \int_0^t  \mathrm{tr}(J)ds} \right).
\]

We conclude then from \eqref{eq-limit-Laplace} after straightforward computations
\end{proof}

\subsection{Asymptotic windings}

We are now interested in the windings of the complex valued process $\det (Z_t)$. We first note that from Theorem \ref{skew}, we have identity in law
\[
\det (Z_t) =\det (I_k+w_t^*w_t)^{-1/2} \det \Theta_t \, \det  \Omega_t.
\]
We shall then use the following result.

\begin{lemma}\label{trace lemma}
Let $G$ be a connected compact Lie group of $m \times m$ matrices with Lie algebra $\mathfrak{g}$. Let $(M_t)_{t \ge 0}$ be a $\mathfrak{g}$-valued continuous semimartingale such that $M_0=0$ and let $(C_t)_{t \ge 0}$ be the $G$-valued solution of the Stratonovich stochastic differential equation
\begin{align*}
dC_t = (\circ d  M_t ) \, C_t 
\end{align*}
Then, for $t \ge 0$, $\det C_t =(\det C_0 )\exp \left( \tr (M_t) \right )$.
\end{lemma}

\begin{proof}
Let $T>0$. Consider on the time interval $[0,T]$ the sequence of $G$-valued  semimartingales $(C_t^n)_{0 \le t \le T}$ inductively defined by
\begin{align*}
C^n_t = C_{t_k}^n \exp \left( \frac{2^n}{T} (t -t_k) (M_{t_{k+1}} -M_{t_k}) \right), \quad t_k \le t \le t_{k+1},
\end{align*}
where $t_k=\frac{kT}{2^n}$, $k =0,\dots, 2^n$. From Theorem 2 in \cite{lepingle}, the sequence of semimartingales $(C_t^n)_{0 \le t \le T}$ converges in probability to $(C_t)_{0 \le t \le T}$ uniformly on $[0,T]$. However,
\begin{align*}
\det (C^n_t) = \det(C_{t_k}^n)  \exp \left( \frac{2^n}{T} (t -t_k) \tr (M_{t_{k+1}} -M_{t_k}) \right), \quad t_k \le t \le t_{k+1}.
\end{align*}
We deduce therefore by induction that 
\[
\det (C^n_T)=(\det C_0 ) \exp \left( \tr (M_T)  \right).
\]
Letting then $n \to +\infty$ yields the conclusion. 
\end{proof}

Using the previous lemma, we deduce the following:

\begin{lemma}\label{Lemma theta}
For every $t \ge 0$, 
$
\det \Theta_t=\frac{\det Z_0}{ | \det Z_0|} \exp \left( \int_{w[0,t]} \mathrm{tr} (\eta) \right).
$
\end{lemma}

\begin{proof}
We have
\begin{align*}
\begin{cases}
d\Theta_t = \circ d \left( \int_{w[0,t]} \eta \right) \, \Theta_t \\
\Theta_0=(Z_0 Z^*_0)^{-1/2}Z_0.
\end{cases}
\end{align*}
Thus from Lemma \ref{trace lemma} we have $\det \Theta_t= \frac{\det Z_0}{ | \det Z_0|} \exp \left( \mathrm{tr} \left( \int_{w[0,t]} \eta \right) \right)$.
\end{proof}

We are now finally in position to prove one of our main results.

\begin{theorem}\label{winding section}
One has the polar decomposition
\[
\det (Z_t)=\varrho_t e^{i\theta_t}
\]
where $0< \varrho_t \le 1$ is a continuous semimartingale  and $\theta_t$ is a continuous martingale such that the following convergence holds in distribution when $t \to +\infty$
\[
\frac{\theta_t}{t} \to \mathcal{C}_{k(n-k)},
\]
where $\mathcal{C}_{k(n-k)}$ is a Cauchy distribution of parameter $k(n-k)$.
\end{theorem}

\begin{proof}
From the decomposition $\det (Z_t) =\det (I_k+w_t^*w_t)^{-1/2} \det \Theta_t \, \det  \Omega_t$ one deduces from Lemmas  \ref{trace lemma} and \ref{Lemma theta} that
\[
\varrho_t =\det(I_k+J_t)^{-1/2}, \, \, i \theta_t = i\theta_0+ \mathrm{tr}(D_t)+ \int_{w[0,t]} \mathrm{tr} (\eta) 
\]
where $D_t$ is a Brownian motion on $\mathfrak{u}(k)$ independent from $w$ and $\theta_0$ is such that $ e^{i\theta_0}=\frac{\det Z_0}{| \det Z_0|}$ . The conclusion follows then from Proposition \ref{thm-main-3}.
\end{proof}

Let us remark that it is also possible to compute the asymptotic law of the radial part $\varrho_t$.   Indeed,  from the previous proof, we know that $\varrho_t =\det(I_k+J_t)^{-1/2}$ and the limit distribution of the ordered eigenvalues of $(I_k-J)(I_k+J)^{-1}$ is computed explicitly in \eqref{limit coulomb} to be a distribution with density
\[
c_{n,k} \prod_{1 \le i < j \le k}(x_i-x_j)^2  \prod_{i=1}^k (1-x_i)^{n-2k}  \, \, \mathbf{1}_{\Delta_k} (x) dx.
\]
Using then the Selberg's integral formula, one obtains that in distribution one has
\[
\varrho_t \to \varrho_\infty,
\]
where $\rho_\infty$ is a random variable such that for every $s \ge 0$
\[
\mathbb{E}( \varrho_\infty^s)=\tilde{c}_{n,k}\prod_{j=0}^{k-1}\frac{\Gamma\left( \frac{s}{2}+j+1 \right)}{\Gamma \left( \frac{s}{2}+n-k+j+1 \right)}
\]
where $\tilde{c}_{n,k}$ is a normalization constant. Thus, using uniqueness of the Mellin transform, one concludes that
\[
 \varrho_\infty^2=\prod_{j=1}^k \mathfrak{B}_{j,n-k}
\]
where $\mathfrak{B}_{j,n-k}$ are independent beta random variables with parameters $(j,n-k)$. This recovers a result by A. Rouault (Proposition 2.4 in \cite{MR2365642}).

\subsection{The case $k \ge n-k$}\label{duality section}

In this section, we prove  Theorem \ref{winding intro} in the case of $k \ge n-k$. Thus, unlike the rest of the paper, we assume in this section that $k \ge n-k$. This is essentially a duality argument equivalent to the isomorphism $G_{n,k} \simeq G_{n,n-k}$. Let
\[
U=\begin{pmatrix}  X & Y \\ Z & V \end{pmatrix}  \in \mathbf{U}(n)
\]
with $Z \in \mathbb{C}^{k \times k}$, $\det (Z) \neq 0$. Using $X^* X+ Z^*Z=I_k$ and $XX^* +YY^*=I_{n-k}$ we deduce the following equality of spectrum
\[
\mathbf{sp} ( Z^* Z)=\mathbf{sp} ( YY^*) \cup \{ 1 \}
\]
and that the eigenvalue 1 of $Z^* Z$ has multiplicity at least $2k-n$. In particular, we have
\[
\det ( Z^* Z)= \det ( Y Y^*)
\]
and
\[
\mathrm{tr}[(Z Z^*)^{-1}-I_k]= \mathrm{tr}[(Y Y^*)^{-1}-I_{n-k}].
\]
Consider now a Brownian motion
\[
U_t=\begin{pmatrix}  X_t & Y_t \\ Z_t & V_t \end{pmatrix}  \in \mathbf{U}(n)
\]
with $Z_0 \in \mathbb{C}^{k \times k}$, $\det (Z_0) \neq 0$. The process $(U_t^*)_{t \ge 0}$ is also a Brownian motion on $\mathbf{U}(n)$. Therefore,  when $t \to +\infty$,
\[
| \det (Z_t)|^2=| \det (Y_t^*)|^2\to \prod_{j=1}^{n-k} \mathfrak{B}_{j,k}
\]
For the study of the winding process of $\det (Z_t)$, we notice that in the case $k \le n-k$ the only part of the proof of Theorem \ref{winding section} that actually uses the fact that $k \le n-k$ is the proof of Theorem \ref{limit_oi} (in the case $k > n-k$ the stochastic differential equation for  $J=w^*w$ is only defined up to the hitting time $\inf \{ t \ge 0, \det J_t =0 \} <+\infty$ and the Girsanov transform method fails). To handle the case $k \ge n-k$, we note that $\mathrm{tr} J=\mathrm{tr}[(Z Z^*)^{-1}-I_k]= \mathrm{tr}[(Y Y^*)^{-1}-I_{n-k}]$. Since the process $(U_t^*)_{t \ge 0}$ is also a Brownian motion on $\mathbf{U}(n)$, one can use the argument of the proof of Theorem \ref{limit_oi} to deduce that
\[
\frac{1}{t^2} \int_0^t  \mathrm{tr}[(Y Y^*)^{-1}-I_{n-k}] ds \to X,
\]
where $X$ is a random variable on $[0,+\infty)$ with density $\frac{k(n-k)}{\sqrt{2\pi}  x^{3/2} }e^{-\frac{k^2(n-k)^2}{2x}}$. Therefore, the conclusion of Theorem \ref{winding section} still holds in the case of $k \ge n-k$.

\subsection{On the moments of $\det Z_t$}

To conclude the paper with a possible view toward free probability that concerns the limit $n \to \infty$ with $k=\alpha n$ (see \cite{MR2384475}) we show that it is possible to give a  formula for the mixed moments of the process $\det Z_t$ that can be deduced from the previous computations.

Let
$$U_t=\begin{pmatrix}  X_t & Y_t \\ Z_t & W_t \end{pmatrix}$$
 be a Brownian motion on the unitary group $\mathbf{U}(n)$ with  $Z_t \in \mathbb{C}^{k\times k}$, $1\le k \le n-1$. Assume that $\det Z_0 \neq 0$.  Let now $p,q \in \mathbb{N}$. As before, from the decomposition $\det (Z_t) =\det (I_k+w_t^*w_t)^{-1/2} \det \Theta_t \, \det  \Omega_t$ one has from Lemmas  \ref{trace lemma} and \ref{Lemma theta} that
\[
\det (Z_t) =\varrho_t e^{i \theta_t}
\]
with
\[
\varrho_t =\det(I_k+J_t)^{-1/2}, \, \, i \theta_t = i\theta_0+ \mathrm{tr}(D_t)+ \int_{w[0,t]} \mathrm{tr} (\eta) 
\]
where $D_t$ is a Brownian motion on $\mathfrak{u}(k)$ independent from $w$ and $\theta_0$ is such that $ e^{i\theta_0}=\frac{\det Z_0}{| \det Z_0|}$.
Therefore, using the martingale of Lemma \ref{martingale g} and the Girsanov transform as in Lemma \ref{kernel J}, we obtain:
\begin{align*}
\mathbb{E}\left( (\det Z_t)^p (\overline{\det Z_t})^q \right)& =\mathbb{E}(\varrho_t^{p+q}  e^{i (p-q) \theta_t})\\
 & =\mathbb{E}(\det(I_k+J_t)^{-(p+q)/2}  e^{i (p-q) \theta_t}) \\
 &=e^{i (p-q) \theta_0 -k (p-q)^2 t} \mathbb{E} \left(\det(I_k+J_t)^{-(p+q)/2}  e^{ (p-q) \int_{w[0,t]} \mathrm{tr} (\eta) }\right) \\
 & = e^{i (p-q) \theta_0 -k (p-q)^2 t} \mathbb{E} \left(\det(I_k+J_t)^{-(p+q)/2}  e^{(p-q)  i \mathcal{B}_{\int_0^t \mathrm{tr}(J) ds}  }\right)   \\
 &=  e^{i (p-q) \theta_0 -k (p-q)^2 t} \mathbb{E} \left(\det(I_k+J_t)^{-(p+q)/2}   e^{-\frac{(p-q)^2}{2} \int_0^t  \mathrm{tr}(J)ds}\right)  \\
 &=  \frac{e^{i (p-q) \theta_0 -k (p-q)^2 t-k(n-k)|p-q|t}}{\det(I+J_0)^{\frac{1}{2}|p-q|}} \mathbb{E}^{\frac{1}{2} |p-q|} \left(\det(I_k+J_t)^{-\min (p,q)}   \right)  
\end{align*}
This last term can be computed because from the proof of Lemma \ref{kernel J}, we have
\begin{align*}
 & \mathbb{E}^{\frac{1}{2} |p-q|} \left(\det(I_k+J_t)^{-\min (p,q)}   \right) \\
 =& C e^{ \left(\frac{1}{3} k(k-1) \left( 3n-4k+3|p-q|+2 \right) \right) t}  \int_{\Delta_k} \mathrm{det} \left( \frac{p^{n-2k,|p-q|}_{t} \left(\frac{1-\lambda_i(0)}{1+\lambda_i(0)},x_j\right)}{(1+x_j)^{-\min(p,q)}}\right)_{ i, j }  \prod_{i>j}(x_i-x_j) \, \, dx.
\end{align*}
where $C$ is a normalization constant depending only on $p,q$ and the ordered eigenvalues of $J_0$, and $p^{n-2k,|p-q|}_{t}$ is given by the formula \eqref{eq-jacobi-kernel}. We note that if $q=0$ the formula simplifies considerably and yields
\begin{align*}
\mathbb{E}\left( (\det Z_t)^p  \right)=(\det Z_0)^p  e^{ -k p^2 t-k(n-k)pt}.
\end{align*}

\bibliographystyle{amsplain}
\bibliography{biblio}

\

\

\begin{itemize}
\item \textsc{F.B: Department of Mathematics, University of Connecticut, Storrs, CT 06269, fabrice.baudoin@uconn.edu}
\item \textsc{J.W: Department of Mathematics, Purdue University, West Lafayette, IN 47907,  jingwang@purdue.edu}

\end{itemize}

\end{document}